\documentclass[10pt]{amsart}
\usepackage{amsmath,amsfonts,amssymb,latexsym,amsthm}
\usepackage{epsfig}
\usepackage{hyperref} 
\hypersetup{colorlinks} 
\usepackage[dvipsnames]{xcolor}
\usepackage{changepage}

\usepackage{dsfont}
\usepackage{mathtools}
\usepackage{fancyhdr}
\usepackage{enumitem}
\usepackage{pgfplots}
\usepackage{url}
\usepackage{tcolorbox}
\usepackage{tkz-euclide}
\usetikzlibrary{graphs}
\usetikzlibrary{graphs.standard}
\usetikzlibrary{arrows}
\usepackage{xcolor}
\usetikzlibrary{arrows.meta}
\usetikzlibrary{graphs}
\usetikzlibrary{graphs.standard}
\usetikzlibrary{arrows}
\usetikzlibrary{cd}

\usepackage{enumitem}

\numberwithin{equation}{section}

\newtheorem{thm}{Theorem}[section]
\newtheorem{lemma}[thm]{Lemma}

\newtheorem{conj}[thm]{Conjecture}
\newtheorem{conv}[thm]{Convention}

\newtheorem{theorem}[thm]{Theorem}

\newtheorem{proposition}[thm]{Proposition}
\newtheorem{corollary}[thm]{Corollary}

\makeatletter
\newcommand{\eqnum}{\leavevmode\hfill\refstepcounter{equation}\textup{\tagform@{\theequation}}}\makeatother

\DeclareMathOperator{\sgn}{sgn}

\title[Complexity of sign imbalance, parity of linear extensions, and height 2 posets]{Complexity of sign imbalance, parity of linear extensions, and height 2 posets}

\author{David Soukup}
\thanks{\thinspace ${\hspace{-.45ex}}^\star$Department of Mathematics,
UCLA, Los Angeles, CA~90095.
\hskip.06cm
Email:
\hskip.06cm
\texttt{soukup@math.ucla.edu}}

\thanks{\today}

\def\<{\langle}
\def\>{\rangle}

\def\0{{\mathbf 0}}

\def\SP{{\textsc{\#P}}}

\def\implies{\Rightarrow}

\def\.{\hskip.06cm}

\usepackage[colorinlistoftodos]{todonotes}

\RequirePackage{cleveref}
\usepackage{hypcap}
\hypersetup{colorlinks=true, citecolor=darkblue, linkcolor=darkblue}
\definecolor{darkblue}{rgb}{0.0,0,0.7}
\newcommand{\darkblue}{\color{darkblue}}

\definecolor{darkred}{rgb}{0.68,0,0}

\definecolor{darkgreen}{rgb}{0,.38,0}

\newcommand{\defn}[1]{\emph{\darkblue #1}}

\begin{document}

\begin{abstract}
Sign imbalance is a statistic on posets which counts the difference between the number of even and odd linear extensions. We prove complexity results about the sign imbalance and parity of linear extensions, focusing on the representative case of height 2 posets. We then consider a recent conjecture of Chan and Pak \cite{CP23+}.
\end{abstract}

\maketitle

\section{Introduction}

Let $P$ be a poset on $n$ elements, and fix some labeling of $P$ with labels $\{1, \dots, n\}$. Then the sign imbalance (defined in section \ref{basics}) is a natural statistic counting the difference between the number of odd and even linear extensions of $P$.

Sign imbalance was introduced by Ruskey in \cite{Rus88} in the context of Gray codes. Define a graph $G(P)$ with vertices corresponding to linear extensions of $P$ and connect pairs of vertices which differ by a transposition. Then it is an easy observation that if $G(P)$ has a Hamiltonian path, then the sign imbalance of $P$ must be at most 1. Furthermore, $G(P)$ is always connected (see \S\ref{geometric}).  The converse was conjectured by Ruskey in \cite{Rus88}. It remains open. Only a small class of special cases have been shown; see \cite[\S5]{Rus03} for a reference or \cite[\S5.5]{Mut23} for a more recent overview. Further information can be found in \cite{Sta05} or \cite{Knu11}. Sign imbalance has also been applied to real algebraic geometry \cite{SS06}  (see \S\ref{geometric}).

Few general results exist for computing the sign imbalance of arbitrary posets. If $P$ is a poset where every nonminimal element is greater than at least two other elements, then $P$ is sign-balanced; switching the labels 1 and 2 provides a bijection between odd and even permutations \cite{Rus88}. Suppose that $P$ is a poset on $n$ elements and that for every maximal chain $C$, the length of $C$ is congruent to $n$ modulo 2. Stanley observed in \cite{Sta05} that the promotion operator provides a sign-reversing involution and so $P$ must be sign-balanced.

Ruskey conjectured that a product of chain posets $C_m \times C_n$ with $m, n > 1$ is sign-balanced if and only if $m \equiv n$ modulo 2 and showed the case where $m, n$ are both even \cite{Rus92}. This conjecture was proven by White \cite{Whi01}, who gave a formula for the case $m \not\equiv n$.  Some other results for specific posets exist (e.g. \cite{Ber18}).

We note that sign imbalance naturally correspond to counting  domino tableaux (see Lemma~\ref{quotients}). For posets arising from Young diagrams, these are the special case of rim hook tableaux where all rim hooks have size 2 with labels that must be increasing along rows and columns.

One problem is to compute the sign imbalance of a poset:

\vspace{2mm}
\begin{tabular}{l l}
\multicolumn{2}{l}{\textsc{\large{Sign Imbalance} } } \\
\textbf{Input:}  & A poset $P$.							\\ 
\textbf{Output:} \quad & the sign imbalance $si(P)$.

\end{tabular}
\vspace{2mm}

Stachowiak gives a complexity result (cf. \S\ref{staiswrong}):

\begin{theorem}[Theorem 1 of \cite{Sta97a}] \textsc{Sign Imbalance} is \textup{ \#\textsf{P}-hard}. This holds even if we consider only posets with height 2.
\end{theorem} 

The proof gives a parsimonious reduction from the sign imbalance of height 2 posets to counting linear extnsions. This corresponds to one direction of Lemma \ref{mainlemma}. Since counting linear extensions was shown to be \#\textsf{P}-hard by Brightwell and Winkler \cite{BW91}, this shows that \textsc{Sign Imbalance} is \#\textsf{P}-hard.

By a theorem of Dittmer and Pak \cite{DP20}, counting linear extensions is still \#\textsf{P}-hard even in the restricted case of  height 2 posets. We prove a complementary result: determining whether a height 2 poset has at least a given sign imbalance is decidable in polynomial time.

\vspace{2mm}

\begin{tabular}{l l}
\multicolumn{2}{l}{\textsc{\large{H2SB} } } \\
\textbf{Input:}  & A poset  $P$ of height $2$ and an integer $k$.							\\ 
\textbf{Decide:} \quad & Is the sign imbalance of $P$ at least $k$?

\end{tabular}

\begin{theorem} \textsc{H2SB} is in \textsf{P}. In the specific case $k = 1$ we obtain that determining whether a height 2 poset is sign balanced is in \textsf{P}.
\end{theorem}

Note that the polynomial bound above implicitly depends on $k$.

Recently, Kravitz and Sah showed in \cite{KS21} an upper bound of $O(\log a \log \log a)$ for the minimal number of elements in a poset with $a$ linear extensions. Lemma~\ref{mainlemma} then allows us to obtain the following corollary:

\begin{corollary}\label{siconcise}
For every positive integer $a$, there exists a height 2 poset $P$ with $O( \log a \log \log a)$ elements such that $si(P) = a$.
\end{corollary}

We contrast Corollary \ref{siconcise} with the following conjecture of  Chan and Pak:

\begin{conj}[Conjecture 5.17 in \cite{CP23+}]\label{chanpakconj} For every sufficiently large integer $m$ there exists a height 2 poset $P$ such that $LE(P) = m$.
\end{conj}

Note that without the ``height 2" condition, Conjecture~\ref{chanpakconj} would be trivial, as $C_{m - 1} + C_1$ has $m$ linear extensions. Furthermore, since a height 2 poset on $n$ elements must have at least $(n/ 2)!^2$ linear extensions, a positive resolution of Conjecture~\ref{chanpakconj} would imply a logarithmic bound similar to that of Conjecture~\ref{siconcise} or \cite{KS21}.

We note that the number of linear extensions of height 2 posets is not equally distributed among odd and even numbers. Let $f(n)$ be the number of height 2 posets on $n$ elements which have an \textit{odd} number of linear extensions.

Given a poset $P =(X, \prec)$, we define  
\begin{align*}
re(P) &:= \#\{i, j \in X : i \prec j \}\\
cr(P) &:= \#\{i, j \in X : i \text{ covers } j \}
\end{align*}
 Equivalently, $re(P)$ is the number of edges in the comparability graph and $cr(P)$ is the number of edges in the Hasse diagram of $P$. Also, for any positive integer $k$ we let $\mathcal{O}_k$ be the set of posets on $k$ elements with an odd number of linear extensions.

\begin{theorem}\label{oddles}
Then for every $n \geq 1$ we have
\[
f(2n + 1) = f(2n) = \sum_{P \in \mathcal{O}_n} 2^{re(P) - cr(P)}
\]

and

\[
2^{\binom{n - 1}{2}} < f(2n) < 2^{\binom{n}{2}}
\]
\end{theorem}
See \S\ref{basicexamples} for an example when $n = 3$. Note that the bounds for $f$ imply that nearly all posets of height 2 have an even number of linear extensions. A similar result holds for all primes:

\begin{theorem}\label{forallq}
Let $q > 1$ be an prime.  Let $f_q(m)$  be the number of height 2 posets $P$ on $m$ vertices such that $q \nmid e(P)$. Then
\[
f_q(m) \leq 2^{\frac{q - 1}{4q}  m^2 + O(m)}
\]
\end{theorem}

For comparison, note that there are $2^{\frac{1}{4}m^2 + O(m)}$ total height 2 posets on $m$ vertices. In the case $q = 2$, Theorems \ref{oddles} and \ref{forallq} agree on an asymptotic $2^{\frac{1}{8} m^2 + O(m)}$.

\section{Definitions and examples}
We use the notation $[n] := \{1, 2, \dots, n\}$. Also, $C_n$ and $A_n$ will denote chain posets and antichain on $n$ elements, respectively. A \emph{lower order ideal} of a poset $P = (X, \prec)$ is a subset $Y \subseteq X$ such that $y \in Y, x \prec y \implies x \in Y$.

\subsection{Posets and linear extensions}\label{basics} We will assume familiarity with basic notions of posets (see e.g. \cite[\S 3]{Sta97} or surveys \cite{BW00, Tro95}). Suppose $P$ is a poset $(X, \prec)$, where $X$ has $n$ elements. Then a \emph{linear extension} of $P$ is a bijection $\ell: X \to [n]$ such that $\ell(x_1) < \ell(x_2)$ whenever $x_1 \prec x_2$. We describe the linear extension as \emph{assigning} the labels in $[n]$ to the elements of $P$. We denote the number of linear extensions of $P$ by $e(P)$ and the set of linear extensions by $LE(P)$.

Fix an arbitrary bijection $f: X \mapsto n$. Then every linear extension corresponds to either an odd or an even permutation. We define the \defn{sign imbalance} as

\begin{equation}\label{sidefn}
si(P) := \left| \sum_{\ell \in LE(P)} \sgn(\ell) \right|
\end{equation}
It is easy to show that $si(P)$ is independent of the choice of $f$.  Thus we will suppress the dependence on $f$. A poset $P$ for which $si(P) = 0$ is called \defn{sign-balanced}. For example, the following poset has $e(P) = 61$ and $si(P) = 1$:

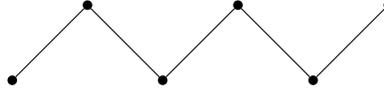
\vspace{5mm}
\begin{center}
\begin{tikzpicture}

\begin{scope}[every node/.style={circle,thick,draw, minimum size = 0cm, inner sep=1pt, fill=black}]
\node (A1) at (0, 5) {};
\node (B1) at (1, 6) {};
\node (A2) at (2, 5) {};
\node (B2) at (3, 6) {};
\node (A3) at (4, 5) {};
\node (B3) at (5, 6) {};
\end{scope}

\path (A1) edge node {} (B1)
(A2) edge node {} (B1)
(A2) edge node {} (B2)
(A3) edge node {} (B2)
(A3) edge node {} (B3);

\end{tikzpicture}
\end{center}
\vspace{5mm}
We note that $e(P) \equiv si(P) \mod 2$ for all posets $P$.

Furthemore, we denote by $P \oplus Q$ the ordinal sum (linear sum) of the posets and by $P + Q$ the disjoint union (parallel sum). Also, a poset is \emph{disconnected} if its Hasse diagram is disconnected.

\subsection{Domino tableaux and quotients} Given a poset $P = (X, \prec)$ with $n$ elements, a \emph{domino tableau} $M$ is a set partition  of $X$ such that:
\begin{enumerate}
\item Every part is a chain of length 2 except for possibly one chain of length 1,
\item If there is a chain of length 1, then it is a maximal element,
\item There exists an ordering $X_1, \ldots, X_k$ of the parts of $M$ such that for all $1 \leq j \leq k$, the set $ X_i \cup \cdots \cup X_j$ is a lower order ideal.
\end{enumerate}

Condition (1) is equivalent to saying that $M$ is a perfect matching in the Hasse diagram plus possibly one extra vertex. We say that a linear extension $\ell$ of $P$ is \emph{adapted} to $M$ when $i$ and $i + 1$ are assigned to same part of $X$ for all odd $i < n$. (If $n$ is odd, then the label $n$ will be assigned to the singleton vertex). Condition (3) implies that there is a linear extension $f \in LE(P)$ adapted to $M$. This can be constructed by assigning $1$ and $2$ to $X_1$, then $3$ and $4$ to $X_2$, and so on. We denote by $DT(P)$ the set of all domino tableaux of $P$

For example, consider the following pair of posets. The left poset has a highlighted domino tableau with an adapted linear extension. The right poset does not admit a domino tableau even though the Hasse diagram does have a perfect matching. Indeed, suppose we match $a$ with $d$ and $b$ with $e$. We cannot put $(a, d)$ before $(b, e)$ because $b \prec d$. And we cannot put $(b, e)$ before $(a, d)$ because $a \prec e$.

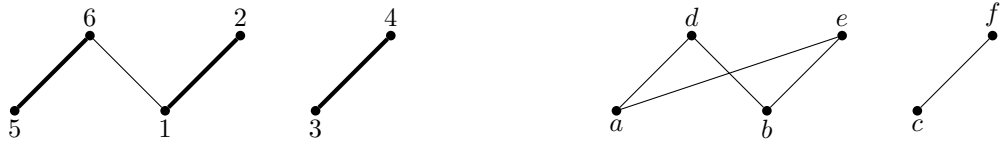
\vspace{5mm}
\begin{center}
\begin{tikzpicture}

\begin{scope}[every node/.style={circle,thick,draw, minimum size = 0cm, inner sep=1pt, fill=black}]
\node (A1) at (0, 5) {};
\node (B1) at (1, 6) {};
\node (A2) at (2, 5) {};
\node (B2) at (3, 6) {};
\node (A3) at (4, 5) {};
\node (B3) at (5, 6) {};

\node (C1) at (8, 5) {};
\node (D1) at (9, 6) {};
\node (C2) at (10, 5) {};
\node (D2) at (11, 6) {};
\node (C3) at (12, 5) {};
\node (D3) at (13, 6) {};

\end{scope}

\path (A1) edge node {} (B1)
(A2) edge node {} (B1)
(A2) edge node {} (B2)
(A3) edge node {} (B3);

\path[ultra thick] (A1) edge node {} (B1)
(A2) edge node {} (B2)
(A3) edge node {} (B3);

\path (C1) edge node {} (D1)
(C2) edge node {} (D1)
(C1) edge node {} (D2)
(C2) edge node {} (D2)
(C3) edge node {} (D3);

\tkzLabelPoint[below](A1){$5$}
\tkzLabelPoint[below](A2){$1$}
\tkzLabelPoint[below](A3){$3$}

\tkzLabelPoint[above](B1){$6$}
\tkzLabelPoint[above](B2){$2$}
\tkzLabelPoint[above](B3){$4$}

\tkzLabelPoint[below](C1){$a$}
\tkzLabelPoint[below](C2){$b$}
\tkzLabelPoint[below](C3){$c$}

\tkzLabelPoint[above](D1){$d$}
\tkzLabelPoint[above](D2){$e$}
\tkzLabelPoint[above](D3){$f$}

\end{tikzpicture}
\end{center}
\vspace{5mm}

It is not hard to show that any two linear extensions adapted to the same domino tableau must have the same sign. Therefore, we define the sign of a domino tableau to be the sign of the linear extensions which are adapted to it.

Given a poset $P$ with a domino tableau $M$, we can construct the \emph{quotient poset} $P/M$ as follows. The vertices of $P/M$ are the elements of $M$. The comparisons of $P/M$ are generated by relations of the form
\[
X_1 \preccurlyeq X_2 \text{ in } P/M \qquad \iff \qquad x_1 \preccurlyeq x_2 \text{ in } P \text{ for some } x_1 \in X_1, x_2 \in X_2
\]

Condition (3) implies that this defines a valid poset structure for $P/M$. As an example, the left poset in the diagram above has $P/M$ isomorphic to $C_2 + C_1$.

\section{Lemmas}

The following lemma is based on a standard involution (see for instance  \cite[Lem. 3]{Rus92}, \cite[Thm. 1]{Sta97a} , \cite[\S5]{Whi01}, and \cite[Corr. 4.2] {Sta05}).

\begin{lemma}\label{quotients}
Let $P$ be a poset. Then 
\begin{equation} \label{sirecursion}
si(P) = \left| \sum_{M \in DT(P)} \sgn(M)\ e(P/M) \right|
\end{equation}
\end{lemma}

\begin{proof}
We construct an involution $\Phi$ on $LE(P)$ where $P = (X, <)$. Suppose $P$ has $n$ vertices and consider a linear extension $\ell$ of $P$. Define the set $S$ to be the set of all odd integers $i \in [n - 1]$ such that $\ell$ assigns $i$ and $i + 1$ to incomparable elements of $P$.

If $S$ is the empty set, then we let $\Phi(\ell) = \ell$. Otherwise, let $j$ be the smallest element of $S$. Then construct $\Phi(\ell)$ by switching the labels $j$ and $j + 1$. By assumption, this is still a valid linear extension. Moreover, it has opposite sign to $\ell$. Therefore the terms corresponding to $\ell$ and $\Phi(\ell)$ will cancel out in the sum (\ref{sidefn}).

Here is an example of $\Phi$. Since $3$ is the smallest odd number not comparable to its successor, it gets switched with $4$.

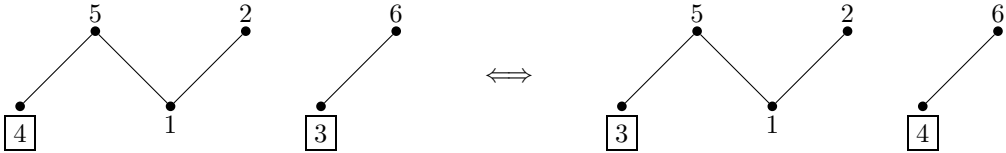
\vspace{5mm}
\begin{center}
\begin{tikzpicture}

\begin{scope}[every node/.style={circle,thick,draw, minimum size = 0cm, inner sep=1pt, fill=black}]
\node (A1) at (0, 5) {};
\node (B1) at (1, 6) {};
\node (A2) at (2, 5) {};
\node (B2) at (3, 6) {};
\node (A3) at (4, 5) {};
\node (B3) at (5, 6) {};

\node (C1) at (8, 5) {};
\node (D1) at (9, 6) {};
\node (C2) at (10, 5) {};
\node (D2) at (11, 6) {};
\node (C3) at (12, 5) {};
\node (D3) at (13, 6) {};

\end{scope}

\path (A1) edge node {} (B1)
(A2) edge node {} (B1)
(A2) edge node {} (B2)
(A3) edge node {} (B3);

\path (C1) edge node {} (D1)
(C2) edge node {} (D1)
(C2) edge node {} (D2)
(C3) edge node {} (D3);

\tkzLabelPoint[below](A1){$\boxed{4}$}
\tkzLabelPoint[below](A2){$1$}
\tkzLabelPoint[below](A3){$\boxed{3}$}

\tkzLabelPoint[above](B1){$5$}
\tkzLabelPoint[above](B2){$2$}
\tkzLabelPoint[above](B3){$6$}

\tkzLabelPoint[below](C1){$\boxed{3}$}
\tkzLabelPoint[below](C2){$1$}
\tkzLabelPoint[below](C3){$\boxed{4}$}

\tkzLabelPoint[above](D1){$5$}
\tkzLabelPoint[above](D2){$2$}
\tkzLabelPoint[above](D3){$6$}

\tkzLabelPoint[above](6.5, 5.2){$\iff$}

\end{tikzpicture}
\end{center}
\vspace{5mm}

We are left only with fixed points of $\Phi$. Suppose $\ell$ is a fixed point. Then $\ell$ is adapted to a unique domino tableau $M$ formed by partitioning $X$ into $\{\ell^{-1}(1), \ell^{-1}(2) \}, \{\ell^{-1}(3), \ell^{-1}(4)\}, \cdots$. This tableau will by definition have $\sgn(M) = \sgn(\ell)$. We can define a linear extension $\ell'$ on $P/M$ by assigning the label $i$ to the subset containing $2i - 1$. This is illustrated in the picture below.

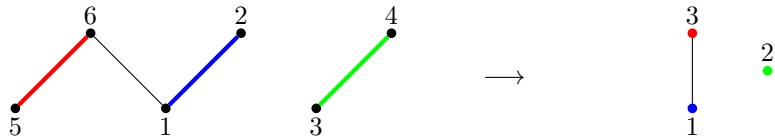
\vspace{5mm}
\begin{center}
\begin{tikzpicture}

\begin{scope}[every node/.style={circle,thick,draw, minimum size = 0cm, inner sep=1pt, fill=black}]
\node (A1) at (0, 5) {};
\node (B1) at (1, 6) {};
\node (A2) at (2, 5) {};
\node (B2) at (3, 6) {};
\node (A3) at (4, 5) {};
\node (B3) at (5, 6) {};

\node[blue] (C1) at (9, 5) {};
\node[red] (D1) at (9, 6) {};
\node[green] (C2) at (10, 5.5) {};

\end{scope}

\path (A1) edge node {} (B1)
(A2) edge node {} (B1)
(A2) edge node {} (B2)
(A3) edge node {} (B3);

\path[ultra thick] (A1) edge[red] node {} (B1)
(A2) edge[blue] node {} (B2)
(A3) edge[green] node {} (B3);

\path (C1) edge node {} (D1);

\tkzLabelPoint[below](A1){$5$}
\tkzLabelPoint[below](A2){$1$}
\tkzLabelPoint[below](A3){$3$}

\tkzLabelPoint[above](B1){$6$}
\tkzLabelPoint[above](B2){$2$}
\tkzLabelPoint[above](B3){$4$}

\tkzLabelPoint[below](C1){$1$}
\tkzLabelPoint[above](D1){$3$}
\tkzLabelPoint[above](C2){$2$}

\tkzLabelPoint[above](6.5, 5.2){$\longrightarrow$}

\end{tikzpicture}
\end{center}
\vspace{5mm}

This constitutes a bijection between linear extensions of $P$ which are adapted to $M$ and $LE(P/M)$. Since every linear extension adapted to a domino tableau is also a fixed point of $\Phi$, we obtain the formula (\ref{sirecursion}) \end{proof}

We give an example of Lemma \ref{sirecursion}.  Consider the poset $P$ below. It has two perfect matchings $\{Y_1, Y_2, Y_3, Y_4\}$ and $\{Z_1, Z_2, Z_3, Z_4\}$, both of which are domino tableaux, illustrated in red and blue below:

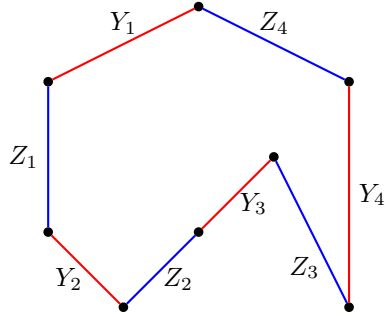
\vspace{5mm}
\begin{center}
\begin{tikzpicture}

\begin{scope}[every node/.style={circle,thick,draw, minimum size = 0cm, inner sep=1pt, fill=black}]
\node (A) at (3, 6) {};
\node (B) at (1, 5) {};
\node (C) at (1, 3) {};
\node (D) at (2, 2) {};
\node (E) at (3, 3) {};
\node (F) at (4, 4) {};
\node (G) at (5, 2) {};
\node (H) at (5, 5) {};
\end{scope}

\path[red, thick] (A) edge node[midway, above, black] {$Y_1$} (B)
(C) edge node[midway, below=2pt, near start,  black] {$Y_2$} (D)
(E) edge node[midway, below=2pt, near end,  black] {$Y_3$} (F)
(G) edge node[midway, right, black] {$Y_4$} (H);

\path[blue,  thick] 
(B) edge node[midway, left, black] {$Z_1$} (C)

(D) edge node[midway, below=2pt, black, near end] {$Z_2$} (E)

(F) edge node[midway, left, near end,  black] {$Z_3$} (G)

(H) edge node[midway, above, black] {$Z_4$} (A);

\end{tikzpicture}
\end{center}
\vspace{5mm}

The two quotient posets (let us call them $Y$ and $Z$) are not isomorphic:

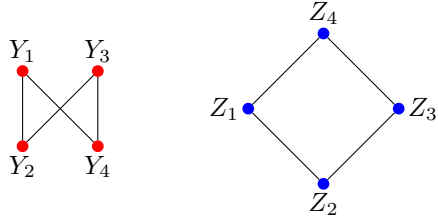
\vspace{5mm}
\begin{center}
\begin{tikzpicture}

\begin{scope}[every node/.style={circle,thick,draw, minimum size = 0cm, inner sep=1pt, fill=black}]
\node[red, ultra thick] (A1) at (1, 2) {};
\node[red, ultra thick] (B1) at (1, 3) {};
\node[red, ultra thick] (C1) at (2, 2) {};
\node[red, ultra thick] (D1) at (2, 3) {};
\node[blue, ultra thick] (E2) at (5, 1.5) {};
\node[blue, ultra thick] (F2) at (4, 2.5) {};
\node[blue, ultra thick] (G2) at (6, 2.5) {};
\node[blue, ultra thick] (H2) at (5, 3.5) {};
\end{scope}

\tkzLabelPoint[below](A1){$Y_2$}
\tkzLabelPoint[above](B1){$Y_1$}
\tkzLabelPoint[below](C1){$Y_4$}
\tkzLabelPoint[above](D1){$Y_3$}

\tkzLabelPoint[below](E2){$Z_2$}
\tkzLabelPoint[left](F2){$Z_1$}
\tkzLabelPoint[right](G2){$Z_3$}
\tkzLabelPoint[above](H2){$Z_4$}

\path (A1) edge node {} (B1)
(A1) edge node {} (D1)
(C1) edge node {} (D1)
(C1) edge node {} (B1);

\path
(E2) edge node {} (F2)

(E2) edge node {} (G2)

(F2) edge node {} (H2)

(G2) edge node {} (H2);

\end{tikzpicture}
\end{center}
\vspace{5mm}
The red quotient $Y$ has $e(P/M) = 4$, and the blue quotient $Z$ has $e(P/M) = 2$. Since the two tableaux have opposite signs, we get $si(P) = |4 - 2| = 2$.

If $P$ has very few domino tableaux, then we can reduce the problem of finding the sign imbalance of $P$ to smaller posets:

\begin{corollary} \label{onedt}
If $P$ is a poset which does not admit a domino tableau, then $P$ is sign-balanced. If $P$ is a poset with a unique domino tableau $M$, then $si(P) = e(P/ M)$.
\end{corollary}

Our goal is to flip Lemma \ref{sirecursion} by building a poset where we control $e(P/M)$. To that end, we will define an operation on posets.

Let $P = (X, \prec)$ be a poset. Call $R$ \emph{good} if it is a subset of $X^2$ with the following properties:
\begin{enumerate}
\item $(x, x) \in R$ for all $x \in X$,
\item $(x, y) \in R$ for all $x,y \in X$ such that $y$ covers $x$ in $P$,
\item If $(x, y) \in R$  then  $x \preccurlyeq y$ in $P$.
\end{enumerate}

In other words, $R$ consists of the diagonal of $X$, all covering relations of $P$, and some subset of the non-covering relations of $P$. Then we define the poset $A(P, R)$ as follows. The vertices consist of pairs of the form $(x, i)$ for $x \in X, i \in \{0, 1\}$. And our relations are given by 

\[
(x, i) \prec (y, j) \text{ if } i = 0, j = 1, \text{ and } (x, y) \in R.
\] 

We claim that this allows us to control the sign imbalance of height 2 posets. We note that the first part of this lemma was proved (in the case where $R$ is maximal) by Stachowiak in \cite{Sta97a}, who used it to show that counting the sign imbalance of a poset is \SP-hard. See \ref{basicexamples} for examples of this construction.

\begin{lemma}[Main lemma]\label{mainlemma}
Let $P$ be a poset. Then for any good $R$ defined as above, $A(P, R)$ is a poset with height 2 and
\[
si(A(P, R)) = e(P).
\]
Conversely, suppose $Q$ is a poset with height 2 that is not sign-balanced. Then if $Q$ has even number of vertices, there exists a poset $P$ and a good set $R$ such that
\[
Q = A(P, R)
\]
And if $Q$ has an odd number of vertices, then there exists a poset $P$ and a good set $R$ such that
\[
Q = A(P, R) + C_1.
\]

\end{lemma}
\begin{proof}
For the first part, note that by construction the Hasse diagram of $A(P, R)$ has only one perfect matching, namely $M = \left\{ \big((x, 0), (x, 1)\big) : x \in X\right\}$. This is also a domino tableau. Since $R$ contains all the covering relations of $P$, we know $A(P, R) / M = P$. The result follows from Corollary \ref{onedt}.

For the other direction, suppose $Q$ is a poset with height 2 which is not sign-balanced. By Corollary \ref{onedt}, $Q$ must have at least one domino tableau $M$. Suppose that $Q$ has an even number of vertices. We claim that the Hasse diagram of $Q$ must in fact have only one perfect matching.

Suppose for contradiction that it had another perfect matching $N$. Then $M \cup N$ must contain at least one cycle of length $ > 2$. As $Q$ was assumed to have height 2, this cycle can only correspond to a subposet of $Q$ which is isomorphic to a crown poset. That is, we have elements $x_1, \dots, x_{2k}$ in $Q$ such that
\[
x_1 \prec_M x_2 \succ_N x_3 \prec_M x_4 \succ \cdots \prec_M x_{2k} \succ_{N} x_1
\]

But this contradicts $M$ being a domino tableau, since we now have a loop in $Q/M$. Therefore $M$ is the unique perfect matching. By construction, every pair of $M$ has a bottom element and a top element.

Now let $P = Q/M$, and define $R \subset M^2$ by 
\[
(M_1, M_2) \in R \iff x \prec y \text{ for some } x \in M_1, y \in M_2.
\]
It is easy to see that $Q = A(P, R)$.

Suppose now that $Q$ has an odd number of vertices. If $Q$ has no isolated vertices, then without loss of generality it has more minimal than maximal elements. But then $Q$ cannot have a domino tableau. So $Q$ must be sign-balanced. (Note that flipping a poset vertically does not affect whether it is sign-balanced). If it has two or more isolated vertices, it also cannot have a domino tableau. So for $Q$ to not be sign-balanced, it must have exactly one isolated vertex $v$. This vertex must be the singleton in the domino tableau, which means $Q - v$ is a height 2 poset which is not sign-balanced. Now we just apply the previous case. \end{proof}

\section{Proofs}

\subsection{Proof of Theorem \ref{oddles}}
First, suppose $P$ is a height 2 poset with an odd number of linear extensions and  $2n + 1$ vertices. Clearly $P$ cannot be sign-balanced. Therefore Lemma \ref{mainlemma} implies that $P$ consists of an isolated vertex and a subposet of $2n$ vertices which also has an odd number of linear extensions. This implies that $f(2n + 1) = f(2n)$. Therefore, we will assume that our posets have an even number of vertices.

Consider the map $A$ from Lemma \ref{mainlemma}. It is easy to see that $A$ is injective. Take the chain poset $C_n$; this clearly has an odd number of linear extensions. Then there are exactly $2^{\binom{n - 1}{2}} $ possible good sets $R$. Thus the set
\[
\{ A(C_n, R) : R \text{ good } \}
\]
establishes our lower bound.

For the upper bound, consider the poset $\Pi_n$ with vertex set $[n] \times \{0, 1\}$ and relations
\[
(a_1, b_1) \prec (a_2, b_2) \text{   if   } a_1 \leq a_2 \text{ and } b_1 < b_2.
\]

Let $\mathcal{D}$ be the set of subposets of $\Pi_n$ such that $(a, 0) \prec (a, 1)$ for all $a \in [n]$. It is clear that the image of $A$ is equal to $\mathcal{D}$ and that $|\mathcal{D}| = 2^{\binom{n}{2}}$. Since any poset with an odd number of linear extensions is not sign-balanced, the lower bound follows. (Note that some posets in the image of $A$ have an even number of linear extensions, so we do not have equality here).  \qed

\subsection{Proof of Theorem \ref{forallq}} Note that it suffices to consider the case where $q$ is prime. We prove a somewhat broader version of \ref{mainlemma}. The proof essentially follows that of \ref{oddles}.

Let $P$ be a height 2 poset on $m$ elements such that $q \nmid e(P)$. Fix $\ell$ to be a linear extension of $P$. Consider the subposets $P_1, P_2, \dots, P_{\lfloor m / q \rfloor}$ defined by
\begin{align*}
P_1 &:= \text{the induced subposet on elements labeled } 1, 2, \dots, q \\
P_2 &:= \text{the induced subposet on elements labeled } q + 1, q + 2, \dots, 2q \\
\cdots
\end{align*}

Note that there may be up to $q - 1$ elements which are not contained in a subposet.  Call $\ell$ \textit{adapted} if none of these subposets are disconnected. (In the case $q = 2$, this means that $\ell$ is adapted to a domino tableau in the sense of Lemma \ref{mainlemma}).

\begin{lemma}\label{easylemma} If $q \nmid e(P)$ then $P$ has a linear extension which is adapted.
\end{lemma}
\begin{proof}
We show that the number of linear extensions which are  not adapted is a multiple of $q$. The following is an equivalence relation on linear extensions which are not adapted.

Given a linear extension $\ell$ which is not adapted, let $
i$ be minimal such that $P_i$ is disconnected. Let $\ell \equiv \ell'$ if $\ell$ and $\ell'$ are equal when restricted to $P \backslash P_i$. It is easy to see that this is an equivalence relation, and the size of an equivalence class is $e(P_i)$.

$P_i$ is disconnected and has $q$ elements. Therefore there exist nonempty posets $Q_1, Q_2$ such that $P_i = Q_1 + Q_2$. But since
\[
e(P_i) = \binom{q}{|Q_1|} e(Q_i) e(Q_2)
\]
we have $q \mid e(P_i)$. This follows because $q$ is prime. That means that the set of linear extensions which are not adapted has been partitioned into equivalence classes of sets each of which has size a multiple of $q$. \end{proof}

So we can fix an adapted linear extension $\ell$ and corresponding subposets $P_1, \dots, P_{\lfloor{m / q} \rfloor}$. By construction, the edges from $P_i$ to $P_j$ where $i < j$ can only be between a bottom vertex of $P_i$ and a top vertex of $P_j$. The notion of bottom and top vertex are well-defined because $P_i$ and $P_j$ are connected. Each subposet can have only at most $q - 1$ elements on the top or bottom.

A simple perturbation argument shows that the greatest number of external edges is possible when the first half of the subposets are of the form $A_{q - 1} \oplus A_1$ and the second half are of the form $A_{1} \oplus A_{q - 1}$. (For simplicity we will assume $\lfloor m / q \rfloor$ is even; extending to the case where $\lfloor m/q \rfloor$ is odd is trivial.) In this case there are at most
\[
\left(\frac{1}{2} \left\lfloor \frac{m}{q} \right\rfloor \right)^2\cdot(q - 1)^2  + \left(\frac{1}{2} \left\lfloor \frac{m}{q} \right\rfloor \right)^2 \cdot (q - 1) + O(m) 
\]
possible external edges. (The first term counts edges which go between the first and second halves, and the second counts edges which remain within one half or the other).

Let $c_q$ be the number of connected height 2 posets with $q$ elements. Then we can construct $P$ as follows: first, we pick each of the $\left\lfloor m/q \right\rfloor$ subposets. Then we add external edges between  the subposets; the above discussion gives a bound on the number of ways to do this. Lastly, each of the $m - \lfloor m / q \rfloor$ remaining vertices can be greater than or incomparable to all the previous vertices. This gives an upper bound of

\[
c_q^{\left\lfloor m/q \right\rfloor} \cdot 2^{\left(\frac{1}{2} \left\lfloor \frac{m}{q} \right\rfloor \right)^2\cdot(q)(q - 1) + O(m)} \cdot 2^{(m - \lfloor m / q \rfloor)} = 2^{\frac{q - 1}{4q}m^2  + O(m)}
\]

possible posets, as required. This completes the proof of Theorem \ref{forallq}.\qed

\section{Final remarks and open problems}

\subsection{Completeness of the number of linear extensions of height 2 posets.}

Our results do not contradict Conjecture \ref{chanpakconj}, but can be used to construct numbers which are not the number of linear extensions of any height 2 poset.

We begin by giving a loose bound on the possible odd numbers of linear extensions of a height 2 poset: 

\begin{proposition}
Let $P$ be a height 2 poset on $2n$ vertices with an odd number of linear extensions. Then
\[
\big( n! \big)^2 \leq e(P) \leq n!(2n - 1)!!
\]

\end{proposition}

\begin{proof}

By Lemma \ref{mainlemma}, any such poset must satisfy

\[
\bigcup_{i = 1}^n C_2 \subseteq P \subseteq A_n \oplus A_n.
\]

\end{proof}

Letting $n = 9$ we see that any height 2 poset on $\leq 18$ elements for which $e(P)$ is odd has at most $9! \cdot 17!! = 12504636144000$ linear extensions. But letting $n = 10$ we see that any height 2 poset on $20$ vertices for which $e(P)$ is odd has at least $10!^2= 13168189440000$ linear extensions. Also by Lemma \ref{mainlemma}, any height 2 poset $P$ on $19$ elements has an isolated vertex, and so $19 \mid e(P)$. Combining these facts we obtain that there is no height 2 poset with $10!^2 - 1 = 13168189439999$ linear extensions. See the footnote of $\S 6.4$ in \cite{CP23+}.

\subsection{Geometric definition of Ruskey's conjecture.}\label{geometric} Given a poset $P$ with elements labeled by $[n]$, the \defn{order polytope} $O(P)$ is the subset of $[0, 1]^n$ given by
\begin{align*}
x_i \leq x_j \text{ whenever } i \prec j \text{ in } P
\end{align*}

Its \defn{canonical triangulation} is given by cutting $O(P)$ with the hyperplanes of the form $x_i = x_j$ for $i, j \in [n]$. Note that each simplex corresponds to a linear extension of $P$. Since this is a triangulation obtained by cutting with hyperplanes, it must be bipartite. \cite{SS06} observed that the two parts correspond exactly to the odd and ever permutations in the definition of sign imbalance, and so a poset is sign-balanced if and only if both parts are the same size.

We note that this provides an immediate proof that the graph $G(P)$ is connected. Any two face-adjacent simplicies in the canonical triangulation of $O(P)$ correspond to linear extensions which differ by an adjacent transposition. Since triangulations must be face-connected,  $G(P)$ is connected.

Soprunova and Sottile \cite{SS06}  considered toric varieties associated to order polytopes. In this case, sign imbalance provides a lower bound for the number of real solutions of a Wronski polynomial system.

\subsection{\textsf{GapP} and \textsf{\#P}} \label{staiswrong}
It is easy to see that sign imbalance is the absolute difference between two \textsf{\#P} functions, one counting even linear extensions and the other counting odd linear extensions. However, this does not imply (as claimed in \cite{Sta97a}) that sign imbalance is in \textsf{\#P}. The closure of \textsf{\#P} under subtraction is called \textsf{GapP} and was defined independently in \cite{FFK94} and \cite{Gup95} building on \cite{OH93}. Therefore, the following conjecture remains open:
\begin{conj}
\textsc{Sign Imbalance} is \emph{not} in \textsf{\#P}.
\end{conj}

Many natural combinatorial differences are not in \textsf{\#P}; see \cite{IP22} for a survey.

\subsection{An example of Theorem \ref{oddles}} \label{basicexamples}Consider the case $n = 3$. There are two posets on 3 elements with an odd number of linear extensions, namely $C_2 + C_1$ and $C_3$. The former poset has $re(C_2 + C_1) = cr(C_2 + C_2) = 1$, and the latter has $re(C_3) = 2$ and $cr(C_3) = 1$. So there are $2^{1- 1} + 2^{2- 1} = 3$ height 2 posets on $6$ elements with an odd number of linear extensions, namely:

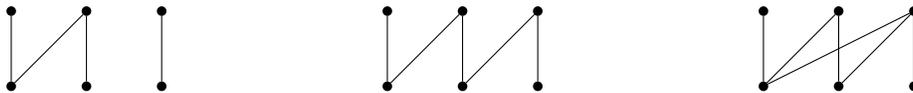
\vspace{5mm}
\begin{center}
\begin{tikzpicture}

\begin{scope}[every node/.style={circle,thick,draw, minimum size = 0cm, inner sep=1pt, fill=black}]
\node (A1) at (0, 5) {};
\node (B1) at (0, 6) {};
\node (A2) at (1, 5) {};
\node (B2) at (1, 6) {};
\node (A3) at (2, 5) {};
\node (B3) at (2, 6) {};

\node (C1) at (5, 5) {};
\node (D1) at (5, 6) {};
\node (C2) at (6, 5) {};
\node (D2) at (6, 6) {};
\node (C3) at (7, 5) {};
\node (D3) at (7, 6) {};

\node (E1) at (10, 5) {};
\node (F1) at (10, 6) {};
\node (E2) at (11, 5) {};
\node (F2) at (11, 6) {};
\node (E3) at (12, 5) {};
\node (F3) at (12, 6) {};

\end{scope}

\path (A1) edge node {} (B1)
(A1) edge node {} (B2)
(A2) edge node {} (B2)
(A3) edge node {} (B3);

\path (C1) edge node {} (D1)
(C2) edge node {} (D2)
(C3) edge node {} (D3)
(C1) edge node {} (D2)
(C2) edge node {} (D3);

\path (E1) edge node {} (F1)
(E2) edge node {} (F2)
(E3) edge node {} (F3)
(E1) edge node {} (F2)
(E1) edge node {} (F3)
(E2) edge node {} (F3);

\end{tikzpicture}
\end{center}
\vspace{5mm}

These posets were obtained by applying the map $A$ in Lemma \ref{mainlemma}. The left poset has $P = C_1 + C_2$, and the other two have $P = C_3$ for different choices of $R$. They have $75$, $61$, and $57$ linear extensions respectively.

\subsection{Euler numbers and posets for which many primes do not divide $e(P)$} The Euler numbers $E_n$ \cite[A000111]{OEIS} count the number of linear extensions of a zizag poset $Z_n$ and have exponential generating function $\sec x + \tan x$. Using these, we can construct posets whose number of linear extensions avoids divisibility by many primes. Theorem \ref{forallq} shows that for all $q$ the number of posets with $q \nmid e(P)$ is small. Here we give an example of an infinite set of posets which satisfy many such conditions.

\begin{proposition} Let $Q$ be a finite set of primes. Then there exists an infinite sequence of posets $P_1, P_2, \dots,$ such that
\[
e(P_i) \equiv 1 \mod q \text{ for all } q \in Q, i \geq 1.
\]
\end{proposition}

\begin{proof}
We show that for any prime $q$ and integer $n > q$, 
\begin{equation}\label{congruences}
E_n \equiv E_q E_{n - (q - 1)} \mod q
\end{equation}

Indeed, consider the set of $\ell \in LE(Z_n)$ such that $1, 2, \dots, q$ are \textit{not} in a contiguous block. For each of these linear extensions, we can permute the labels $1, 2, \dots, q$ in some number of ways that is divisible by $q$. (This is essentially the same as the proof of Lemma \ref{easylemma}). So, $\mod q$, we can ignore these. But if $1, \dots, q$ are in a contiguous block, then we can collapse the entire block into a single element; this leaves $Z_{n - (q - 1)}$. This shows \ref{congruences}.

Since $E_q \equiv \pm 1 \mod q$ \cite[A000111]{OEIS} and $E_1 = 1$ we need only pick $n$ such that $n \equiv 1 \mod (q - 1)$ for all $q \in Q$. \end{proof}

As a side note, we show that $3 \nmid E_n$ for all $n \geq 1$. Indeed, one can verify that $3 \nmid E_n$ for all $1 \leq n \leq 3$, and by \ref{congruences} the result follows.  Similarly, computer calculations show that $p \nmid E_n$ for all $n$ for 
\[
p = 3, 7, 11, 23, 83, 107, 163, 167,
179,
191,
199,
211,
227,
239,
367,
383,
443,
479,
487,
503,
599, \dots
\]

It is not clear if there are infinitely many such primes.

\subsection{Acknowledgments}

I would like to thank my advisor Igor Pak for suggestions and guidance, as well as Andrew Sack for helpful discussions.

\end{document}